\documentclass[a4paper, 12pt] {article}
\usepackage[cp1251]{inputenc}
\usepackage[english] {babel}
\usepackage {amsmath}
\usepackage {amssymb,amsthm}
\usepackage {bbm}
\usepackage {latexsym}
\usepackage[colorinlistoftodos, textwidth=4cm, shadow]{todonotes}
\usepackage{amscd}
\usepackage[matrix,arrow,curve]{xy}

\newcommand {\protimes} {\mathbin{\widehat{\otimes}}}
\theoremstyle{definition}
\newtheorem{Theorem}{Theorem}
\newtheorem{Prop}{Proposition}

\newtheorem{Def}{Definition}

\newtheorem{Lemm}{Lemma}
\def\phi{\varphi}

\begin{document}
\title{A non-Archimedean counterpart of Johnson's theorem for discrete groups}
\author{Yuri Kuzmenko}
\date{}

\maketitle
\abstract{Let $K$ be a spherically complete field with a non-Archimedean valuation. We define a new version of $K-$amenability for discrete groups and show that the Banach $K-$algebra $l^1(G)$ is amenable iff $G$ is $K-$amenable in our sense.

\section{Introduction}
A Banach algebra $A$ is called amenable if for any $A-$bimodule $X$ every continious derivation with values in $X^*$ is inner. The term ``amenable'' is justified 
 by B.~Johnson's theorem: A locally compact group $G$ is amenable iff the group algebra $L^1(G)$ is amenable. The aim of this paper is to prove a counterpart of Johnson's theorem in the non-Archimedean case for discrete groups. 


The notion of amenability for Banach algebras was intoduced by B.Johnson~\cite{Johnson} in 1972. A number of homological characterizations of amenable Banach algebras were obtained by Helemskii and Sheinberg~\cite{HelSheinberg} in 1979.  

Let $K$ be a spherically complete field with a non-Archimedean valuation, and $G$ be a locally compact group. 
Amenable algebras over $K$ are defined in the same way as amenable algebras over $\mathbb C$ (see below for details). 
On the other hand, the  traditional definition of an amenable group in terms of invariant means cannot be transferred verbatim from the classical case. A possible way to define a $K-$amenable group was suggested by W.Schikhof in \cite{NarchAm}. He also obtained a characterization of $K-$amenable groups in structure terms. 
However, the amenability of $l^1(G)$ is not equivalent to the Schikhof $K-$amenability of $G$ even in the case where $G$ is finite. For example, if $K=\mathbb Q_p$, then $\mathbb Z/p\mathbb Z$ is not Schikhof $K-$amenable, while $l^1(\mathbb Z/p\mathbb Z)$ is a amenable algebra (see Section \ref{result} for details). 
 
In this paper we give an alternative definition of $K-$amenability for discrete groups, namely Johnson amenability (see Definition~\ref{DefJAmenable} below). Similarly to Schikhof $K-$amenability, Johnson amenability reduces to the classical one when we replace $K$ by $\mathbb C$. Our main result states that $l^1(G)$ is amenable iff $G$ is Johnson $K-$amenable.

\section{Preliminaries}
For basic notions of non-Archimedean functional analysis we refer to \cite{LCS_NArch} and \cite{vanRooj}. Unless otherwise specified, we let $K$ be a spherically complete field with a non-Archimedean valuation.
\subsection{Preliminaries from relative homology \\ theory of Banach algebras}
We refer to \cite{HelHomBanEng} for basic notions in the homology theory of Banach algebras over $\mathbb C$. Most definitions of this theory make sense for Banach algebras over $K$. For the reader's convenience, we recall some definitions below.


Let $A$ be any Banach algebra over $K$. All $A-$modules that we consider in the present paper are Banach $A-$modules, so for simplicity we always write ``$A-$module'' for ``Banach $A-$module''. If $E$ and $F$ are Banach spaces, then $E\protimes F$ stands for their completed tensor product (\cite[chapter 4]{vanRooj}. 
If $M$ is a right $A-$module and $N$ is a left $A-$module, then $M\protimes_A N$ stands for their $A-$module tensor product, i.e., the quotient $(M\protimes N)/J$, where $J$ is the closed linear span of all elements of the form $m\cdot a\otimes n-m\otimes a\cdot n$ ($m\in M, n\in N$, and $a\in A$). Let $A_+$ denote the unitization of $A$. If $A$ is unital, then each unital $A-$bimodule can be regarded as an $A^e-$module, where $A^e=A\protimes A^{op}$ and $A^{op}$ is the algebra opposite to $A$.

A complex of $A-$modules is called {\itshape admissible} if it splits as a complex of Banach spaces. A right $A-$module $M$ is called {\itshape flat} if the functor $M\protimes_A (-)$ takes 
admissible complexes to exact complexes. Flat left modules and flat bimodules are defined similarly. 

\begin{Def}
A Banach algebra $A$ called {\itshape amenable} if $A_+$ is flat as an $A-$bimodule.
\end{Def}

Let $\pi \colon A_+\protimes A_+ \to A_+$ and $\pi_0\colon A\protimes A^{op}\to A$ denote the multiplication maps given by $a\otimes b \mapsto ab$. We denote the kernel of $\pi$ by $I^{\Delta}$. 
An $A-$module morphism is called an {\itshape admissible monomorphism} if it has a left inverse as a mapping of Banach spaces.
\begin{Def}
An $A-$module $I$ is called {\itshape injective} if for each admissible $A-$module monomorphism $i\colon Y\to X$ and each $A-$module morphism $\varphi\colon Y\to I$ there exists a morphism $\psi\colon X\to I$ such that the following diagram is commutative: 
$$
\xymatrix{
& X\ar@{-->}[dl]_{\psi}\\
I& Y\ar[l]^{\varphi}\ar[u]_{i}
}.
$$
\end{Def}

The following propositions are proved exactly as in the Archimedean case (cf. \cite[\S VII.2]{HelHomBanEng}). 
\begin{Prop}
The following properties of a Banach algebra $A$ are equivalent
	\begin{enumerate}
	\item $A$ is amenable.
	\item For any $A-$bimodule $X$, every continious derivation of $A$ with values in $X^*$ is inner (i.e., for any derivation $D$ there exists an element $a\in A$ such that for all $x\in X^*$ we have $D(x)=ax-xa$). 
	\item $A$ is flat as an $A-$bimodule and has a bounded approximate identity.
	\item $I^{\Delta}$ has a right bounded approximate identity.
	\item $\pi^*\colon A^*_+\to(A_+\protimes A_+)^*$ is an $A-$bimodule coretraction.
	\item There is a virtual diagonal in $(A\protimes A^{op})^{**}$, i.e., an element $m\in (A\protimes A^{op})^{**}$ such that for any $a\in A$ we have $$(a\otimes 1)m=(1\otimes a)m,$$ $$\pi_0^{**}(m) a=a\pi_0^{**}(m)=a.$$ \label{VirtDiag}
	\end{enumerate}
\end{Prop}


\begin{Prop}\label{1.13}
If $M$ is an injective module and $\varphi\colon M\to N$ is an admissible monomorphism, then $\varphi$ has a left inverse $A-$module morphism.
\end{Prop}
\begin{Prop}\label{1.14}
$M$ is a flat left $A-$module iff $M^*$ is an injective right $A-$module.
\end{Prop}

\begin{Prop}\label{flat}
If $A$ is an amenable Banach algebra, then each right (left) $A-$module is flat.
\end{Prop}

The proof of the next proposition is also almost the same as in the Archimedean case, but for the convenience of the reader we will give it. 
\begin{Prop}\label{1.20}
Let $I$ be a closed left ideal in $A_+$. Suppose the complex of right $A-$modules $J=(0\to (A_+/I)^*\stackrel{j^*}{\longrightarrow} A_+^*\stackrel{i^*}{\longrightarrow} I^*\to 0)$ is admissible, where $i$ and $j$ are the natural injection and projection respectively. Then the following conditions are equivalent:
\begin{enumerate}
\item $I$ has a right bounded approximate identity.
\item $J$ splits in $\text{mod}-A$.
\item $A_+/I$ is a flat left $A-$module.
\item $(A_+/I)^*$ is an injective right $A-$module.
\end{enumerate}
\end{Prop}
\begin{proof}
From Propositions \ref{1.13} and \ref{1.14} we see that conditions (2), (3) and (4) are equivalent.

$(2)\Rightarrow (1)\colon$ Let $\rho$ be a right inverse morphism to $i^*$. Then $\rho^{*}\colon A_+^{**}\to I^{**}$ is a left $A-$module morphism and $\rho^{*}(a)=a$ for all $a\in I^{**}$. Let $\hat e=\rho^*(e)$, where $e$ is the identity in $A_+\subset A_+^{**}$. By Goldstine's theorem (\cite[Corollary 7.4.8]{LCS_NArch}), there exists a bounded net $e_\nu$ that converges to $\hat e$ in the weak$^*$ topology. For any $a\in A$ the multiplication by $a$ is continuous on $I^{**}$ in the weak$^*$ topology. Hence the net $a\cdot e_\nu$ converges to $a\cdot \hat e=a\cdot \rho^*(e)=\rho^*(a\cdot e)=\rho^*(a)$. If $a\in I\subset I^{**}$, then $\rho^*(a)=a$. So $e_{\nu}$ is a weak right bounded approximate idenitity in $I$. By \cite[\S 11, Proposition 4]{Bonsall-Duncan}, $I$ has a right bounded approximate identity.

$(1)\Rightarrow (4)$ is a corollary from \cite[Theorem VII.1.5]{HelHomBanEng}.
\end{proof}

\subsection{Preliminaries from Hopf algebra theory}
Let us recall some definitions and facts on Hopf Banach algebras. For details, see \cite{CartanEilenberg} (for the algebraic case), \cite[Section 2]{PirkStbFlat} (for Hopf Banach algebras) and \cite{Diarra} (for non-Archimedean Hopf Banach algebras)

\begin{Def}
A Hopf Banach algebra is a unital Banach $K-$algebra $A$ together with algebra homomorphisms
\begin{align*}
\Delta \colon A\to A\protimes A &\text{ (comultiplication),}\\
\varepsilon \colon A\to K &\text{ (counit),}\\
\end{align*}
and an algebra antihomomorphism
\begin{align*}
S \colon A\to A &\text{ (antipode)}
\end{align*}
such that the following diagrams are commutative:
$$
\xymatrix{
A\protimes A\protimes A & A\protimes A \ar[l]_-{\Delta\protimes \text{id}_A} & & K\protimes A & A\protimes A \ar[l]_{\varepsilon\protimes \text{id}_A} \ar[r]^{\text{id}_A\protimes \varepsilon}& A\protimes K\\
A\protimes A \ar[u]^{\text{id}_A\protimes\Delta} & A \ar[u]_{\Delta} \ar[l]_{\Delta} & & & A\ar[lu]^{l_A^{-1}} \ar[u]_{\Delta} \ar[ru]_{r_A^{-1}} & 
}
$$
$$
\xymatrix{
A\protimes A \ar[d]_{S\protimes \text{\text{id}}_A} & A \ar[l]_-{\Delta} \ar[r]^-{\Delta} \ar[d]^{\nu\varepsilon} & A\protimes A \ar[d]^{\text{id}_A\protimes S}\\
A\protimes A \ar[r]^-{m}& A & A\protimes A \ar[l]_-{m}
}
$$
where $m\colon A\protimes A\to A$ is the algebra multiplication, $\nu \colon K\to A$ is defined by $1\mapsto 1_A$, $\text{id}_A$ is the identity morphism, $r_A\colon A\protimes K\to A$ and $l_A\colon K\protimes A\to A$ are the isomorphims defined by $a\otimes 1\mapsto a$ and $1\otimes a\mapsto a$, respectively ($a\in A$). 
\end{Def}


Consider the algebra morphism $E\colon A\to A^e, \quad E=(1\otimes S)\Delta$. It defines a right $A-$module structure on $A^e$ by $u\cdot a=uE(a)$ ($u\in A^e, a\in A$). We denote this module by $A_E^e$. Thus $A^e_E$ is an $A^e$-$A-$bimodule with respect to the left action of $A^e$ given by multiplication. We also consider $K$ as a left $A-$module via $\varepsilon$.

\begin{Lemm}
\label{$A_E^e$}
Let $A$ be a Hopf Banach algebra with invertible antipode. Then there is an isomorphism of left $A^e-$modules $$A^e_E\protimes_A K\to A,\quad u\protimes1\mapsto m(u).$$
\end{Lemm}
The proof of this lemma is the same as in the case $K=\mathbb{C}$ (see ~\cite[Lemma 2.4]{PirkStbFlat}).

\section{Results} \label{result}
Let $G$ be a discrete group, and let $l(G)=l^1(G)$ denote the space of all $K-$valued functions on $G$ that vanish at infinity. Each such function $f$ can be uniquely decomposed as $f=\sum_{g\in G} \alpha_g \delta_g$ with $\alpha_g\in K$, where $\delta_g$ is the function which is $1$ at $g$ and $0$ elsewhere. In what follows, we often identify $\delta_g$ with $g$. It is well known that $l(G)$ is a Banach space with respect to the norm $\|f\|=\text{max}|\alpha_g|$. Moreover, $l(G)$ is a Banach algebra under convolution, since 
\begin{multline*}
\|\sum_{g\in G}\alpha_g g\sum_{g\in G}\beta_g g\|=\max_g |\sum_{t\in G} \alpha_t \beta_{t^{-1}g}|\leq \\
\leq \max_{g,t} |\alpha_t\beta_{t^{-1}g}|\leq
\|\sum_{g\in G}\alpha_g g\| \|\sum_{g\in G}\beta_g g\|.
\end{multline*}
Observe that $\delta_e$ is the identity in $l(G)$. Let $I_0$ be the ideal of functions $f=\sum_g \alpha_g g$ such that $\sum_{g\in G}\alpha_g=0$. We denote by $\mathbbm{1}$ the function that is identically $1$ on $G$. 

We can define a left action of $G$ on $l^\infty (G)$ by $(g\cdot f)(h)=f(g^{-1}h)$.

It is easy to see that $l(G)$ has a Hopf Banach algebra structure uniquely determined by 

\begin{align*}
\Delta\colon \delta_x\mapsto \delta_x\otimes\delta_x,\\
\varepsilon(f)=\sum_{g\in G}f(g),\\
Sf(g)=f(g^{-1}).
\end{align*}

\begin{Def}\label{DefJAmenable}
We say that the group $G$ is {\itshape Johnson $K-$amenable} if there exists a left-invariant $m\in (l^{\infty}(G))^*$ (i.e., $m(g\cdot f)=m(f)$ for any $g\in G$ and $f\in l^{\infty}(G)$) such that $m(\mathbbm{1})\neq 0$.
\end{Def}
Observe that if we replace $K$ by $\mathbb{C}$, then the Johnson amenability of $G$ is equivalent to the usual amenability. Indeed, take the measure on $G$ associated with $m$, and consider the functional associated with the variation of this measure. The resulting functional $|m|$ is clearly positive, left invariant, and $|m|(\mathbbm{1})\ne 0$. Hence $|m|/(|m|\mathbbm{1})$ is a left invariant mean.

\begin{Lemm} \label{mor}
Let $A$ be a Hopf Banach algebra. There is a morphism of left $A^e-$modules $$y\colon A^e_E\protimes_A A^{**}\to (A^e)^{**}$$ uniquely determined by $$u\otimes\beta\mapsto uE^{**}(\beta)\quad (u\in A^e, \beta\in A^{**}).$$
\end{Lemm}
\begin{proof}
Let $Y\colon A^e\times A^{**}\to (A^e)^{**}$ be the bilinear mapping defined by $$(u,\beta)\mapsto uE^{**}(\beta)\quad (u\in A^e,\; \beta\in A^{**}).$$ Let us prove that $Y$ defines a mapping $A^e_E\protimes_A A^{**}\to (A^e)^{**}$. It is sufficient to show that $$E^{**}(c\beta)=E(c)E^{**}(\beta)\quad (c\in A,\; \beta\in A^{**}).$$ We clearly have $$E^{**}(c\beta)(\varphi)=\beta(E^*(\varphi)\cdot c)\quad (c\in A,\; \beta\in A^{**}, \varphi\in (A^e)^*).$$ and $$(E(c)E^{**}(\beta))(\varphi)=E^{**}(\beta)(\varphi \cdot E(c))=\beta(E^*(\varphi\cdot E(c)))\quad (\varphi\in (A^e)^*).$$
Hence it is sufficient to prove that 
\begin{equation}
E^*(\varphi)\cdot c=E^*(\varphi\cdot E(c)).
\end{equation}
We have 
$$ (E^*(\varphi)\cdot c)(a)= E^*(\varphi)(ca) =\varphi(E(ca)),$$
and
$$ E^*(\varphi\cdot E(c))(a) = (\varphi\cdot E(c))(E(a)) =\varphi(E(c)E(a))=\varphi(E(ca))\quad (a\in A). $$
This implies (1), which in turn implies that $Y$ induces a continuous linear map $y$. Clearly, $y$ is a left $A^e-$module morphism.
\end{proof}

\begin{Theorem}
$l(G)$ is amenable iff $G$ is Johnson $K-$amenable.
\end{Theorem}
\begin{proof}
 Suppose $l(G)$ is amenable. By Propositions \ref{1.14} and \ref{flat}, $K\cong K^*$ is injective as a right $l(G)-$module. So Proposition \ref{1.13} implies that the mapping $K \hookrightarrow l^{\infty}(G)$ defined by $1\mapsto \mathbbm{1}$ has a left inverse $l(G)-$module morphism $f\colon\ell^\infty(G)\to K$.  Given $h\in l(G)$, define $\tilde h\in l(G)$ by $\tilde h(s)=h(s^{-1})$. For each $\varphi\in l^\infty(G)$ and each $h\in l(G)$ we have $\varphi\cdot h=\tilde h\star\varphi$. Hence for each $g\in G$ we have $g\varphi=\delta_g\star\varphi=\varphi\cdot\delta_{g^{-1}}$. Then $$f(g\varphi)=f(\varphi\cdot \delta_{g^{-1}})=f(\varphi)\cdot \delta_{g^{-1}}=f(\varphi),$$ for each $\varphi\in l^\infty(G)$, i.e., the functional $f$ is left invariant. Clearly, $f(\mathbbm{1})\neq 0$. Hence $G$ is Johnson $K-$amenable.

 Conversely, suppose $G$ is Johnson $K-$amenable, i.e., there exists a left invariant functional $f$ on $l^\infty (G)$ such that $f(\mathbbm{1})\ne0$. Let $m=f(\mathbbm{1})^{-1}f$. 
Then there exists a left $l(G)-$module morphism $z\colon l(G)/I_0\cong K\to l(G)^{**}$ defined by $1\mapsto m$.

To prove that $l(G)$ is amenable, it is sufficient to construct a morphism of $l(G)-$bimodules $d\colon l(G)\to (l(G)\protimes l(G)^{op})^{**}$ such that for all $a\in l(G)$ we have $a\cdot \pi_0^{**}(d(\delta_e))=\pi_0^{**}(d(\delta_e))\cdot a=a$. Indeed, if $d$ satisfies the above conditions, then $d(\delta_e)$ is a virtual diagonal, so $l(G)$ is an amenable algebra by Proposition 1.

Consider the image of $z$ under the functor $l(G)^e_E\protimes_{l(G)}(-)$: $$\tilde z\colon l(G)^e_E\protimes_{l(G)} K\to l(G)^e_E\protimes_{l(G)} l(G)^{**}.$$ By Lemma~\ref{$A_E^e$} we have a left $l(G)^e$-module isomorphism
$$l(G)^e_E\protimes_{l(G)} K\cong l(G).$$ 
Let 
$$y\colon l(G)^e_E\protimes_{l(G)}l(G)^{**}\to (l(G)\protimes l(G)^{op})^{**}$$ be the morphism constructed in Lemma~\ref{mor}. Hence we obtain a morphism $$d=y\tilde z\colon l(G)\to (l(G)\protimes l(G)^{op})^{**}.$$ 
Let $\nu$ and $\varepsilon$ denote the unit and the counit of $l(G)$, respectively. We have
\begin{multline*}
\pi_0^{**}(d(\delta_e))=\pi_0^{**}(y(\delta_e\otimes\delta_e\otimes m))=\pi_0^{**}((\delta_e\otimes\delta_e)E^{**}(m))=\\ =\pi_0^{**}(E^{**}(m))=(\pi_0 E)^{**}(m)=(\nu\varepsilon)^{**}(m)=\nu^{**}\varepsilon^{**}(m)=m(\mathbbm{1}) \delta_e=\delta_e.
\end{multline*}
This clearly implies that $\pi_0^{**}(d(\delta_e))\cdot a=a\cdot \pi_0^{**}(d(\delta_e))=a$ for all $a\in l(G)$, whence $d(\delta_e)$ is a virtual diagonal. This completes the proof.
\end{proof}

In conclusion, let us say a few words about amenability  in the sense of Schikhof. In \cite{NarchAm}, Schikhof gave the following definition.
\begin{Def}\label{SchAm}
A locally compact group $G$ is {\itshape $K-$amenable} if there exists a $K-$linear functional $m\colon C_b(G)\to K$ such that $m$ is left-invariant, $m(\mathbbm{1})=1$, and $\|m\|\leq 1$.

\end{Def}
To avoid confusion, we say that $G$ is {\itshape Schikhof $K-$amenable} if it is $K-$amenable  in the sense of Definition~\ref{SchAm}. Schikhof have shown that the group $G$ is Schikhof $K-$amenable iff it satisfies the follwing conditions:
\begin{enumerate}
\item For each finite set $S\subset G$ there exists a compact subgroup $N\subset G$ such that $S\subset N$. \label{torsional}
\item For every pair of open subgroups $S_1\subset S_2$ the number $\left[S_2\colon S_1\right]$ is not divisible by the characteristic of $k$ (where $k$ is the residue class field of $K$).\label{char-free}
\end{enumerate}
The following result is in fact contained in Schikhof's proof of the above criterion.
\begin{Theorem}[Schikhof]
A discrete group $G$ is Johnson $K-$amenable iff it satisfies condition~\ref{torsional}.
\end{Theorem}

Observe that Schikhof amenability is a stronger condition than Johnson amenability. For example, if $K=\mathbb Q_p$ and $G=\mathbb Z/p\mathbb Z$, then $G$ satisfies condition \ref{torsional} but does not satisfy condition \ref{char-free}.



\end{document}